\documentclass[12pt,reqno]{amsproc}

\usepackage[margin=.9in]{geometry}
\usepackage[T1]{fontenc}
\usepackage{bbm}
\usepackage[usenames]{color}
\usepackage{mathpazo}
\usepackage{amsmath,amssymb,amsthm} 
\usepackage{wrapfig}
\usepackage{tikz-cd}
\usepackage[shortlabels]{enumitem}
\usepackage{blindtext}
\usepackage{caption}
\usepackage{mathrsfs,bm,nicefrac}
\usepackage[all,cmtip]{xy}

\definecolor{PineGreen}{rgb}{0.0,0.47,0.44}
\definecolor{MidnightBlue}{rgb}{0.1,0.1,0.44}
\definecolor{magenta}{rgb}{1.0,0.0,1.0}
\definecolor{bl1}{HTML}{4479A1}
\definecolor{pur1}{HTML}{52196D}
\definecolor{mag1}{HTML}{2AD0F1}
\definecolor{org1}{rgb}{.92,.39.21}
\definecolor{pur2}{rgb}{.53,.47,.7}

\usepackage{hyperref}
\hypersetup{
	colorlinks=true,
	linkcolor=bl1,
	citecolor=PineGreen,
	filecolor=magenta,
	urlcolor=bl1
}
\usepackage{cleveref}

\newcommand{\spc}{\hspace*{0.25in}}

\makeatletter
\newcommand{\eqnum}{\refstepcounter{equation}\textup{\tagform@{\theequation}}}
\makeatother

\newtheorem{theorem}{Theorem}
\numberwithin{theorem}{section}
\newtheorem{proposition}[theorem]{Proposition}
\newtheorem*{theorem*}{Theorem}
\newtheorem{lemma}[theorem]{Lemma}
\newtheorem{corollary}[theorem]{Corollary}
\theoremstyle{definition}

\newtheorem{definition}[theorem]{Definition}

\theoremstyle{remark}
\newtheorem{remark}[theorem]{Remark}
\newtheorem{example}[theorem]{Example}

\newcommand{\Con}{\mathbf{Con}}
\newcommand{\Gr}{\mathbf{Gr}}

\newcommand{\RR}{\mathbb{R}}

\newcommand{\QQ}{\mathbb{Q}}

\newcommand{\PP}{\mathbb{P}}

\newcommand{\CC}{\mathbb{C}}

\DeclareMathOperator{\codim}{codim}

\newcommand{\bV}{\mathbf{V}}

\newcommand{\cO}{\mathscr{O}}

\newcommand{\jac}{\text{\bf J}}
\usepackage[normalem]{ulem}
\usepackage{subcaption}
\makeatletter
\DeclareRobustCommand
\myvdots{\vbox{\baselineskip4\p@ \lineskiplimit\z@
		\hbox{.}\hbox{.}\hbox{.}}}
\makeatother

\newcommand{\sing}[1]{#1_{\text{sing}}}
\newcommand{\reg}[1]{#1_{\text{reg}}}
\newcommand{\xx}{\mathbf{x}}
\newcommand{\yy}{\mathbf{y}}

\newcommand{\uu}{\mathbf{u}}
\newcommand{\setof}[2]{\{#1\;|\;#2\}}
\newcommand{\sat}[2]{(#1:#2^{\infty})}
\newcommand{\WS}{W_{\bullet}}
\newcommand{\mthcall}[1]{\text{\bf #1}}
\newcommand{\field}{\mathbb{K}}
\newcommand{\mord}{\prec}

\DeclareMathOperator{\lcm}{lcm}
\DeclareMathOperator{\lc}{lc}
\DeclareMathOperator{\lm}{lm}
\DeclareMathOperator{\ass}{Assoc}
\DeclareMathOperator{\mon}{Mon}
\DeclareMathOperator{\gr}{gr}

\usepackage[foot]{amsaddr}

\def\DD{D\kern-.7em\raise0.3ex\hbox{\char '55}\kern.33em}

\usepackage{booktabs}
\usepackage{colortbl}
\definecolor{Ftitle}{RGB}{11,46,108}
\colorlet{tableheadcolor}{Ftitle!25} 
\colorlet{tablerowcolor}{gray!10} 
\begin{document}\setlength\parindent{15pt}
	
	\title{Faster Computation of Whitney Stratifications and Their Minimization}
	\author{Martin Helmer} 
	\address[MH]{
		Department of Mathematics, Swansea University,
		Swansea, Wales, UK}\email{martin.helmer@swansea.ac.uk}
              \author{Rafael Mohr}\address[RM]{Department of Computer Science, KU Leuven,
              Leuven, Belgium}\email{rafaeldavid.mohr@kuleuven.be}\thanks{The second author is the corresponding author.}
	\maketitle
\begin{abstract}
  We describe two new algorithms for the computation of Whitney
  stratifications of real and complex algebraic varieties. The first
  algorithm is a modification of the algorithm of Helmer and Nanda
  (HN) \cite{hnFOCM,hn2Real}, but is made more efficient by using
  techniques for equidimensional decomposition rather than computing
  the set of associated primes of a polynomial ideal at a key step in
  the HN algorithm. We note that this modified algorithm may fail to
  produce a minimal Whitney stratification even when the HN algorithm
  would produce a minimal stratification.

  The second algorithm coarsens a given Whitney stratification of a
  complex variety to the unique minimal Whitney stratification; we refer to this as the {\em minimization} of a stratification. The theoretical
  basis for our approach is a classical result of Teissier. To our
  knowledge this yields the first algorithm for computing a minimal
  Whitney stratification.

 
\end{abstract}

\section{Introduction} First introduced by Whitney in
\cite{whitney1965tangents}, the concept of {\em Whitney
  stratification} is now an essential tool in the study of singular
spaces. This construction provides the basic structure needed to
decompose singular spaces into smooth manifolds which join together in
a desirable way. In this note we present a new method to compute such
statifications which substantially improves on the state of the art in
many cases, building on the recent works of \cite{hnFOCM,hn2Real} by
incorporating ideas used in symbolic algorithms which compute
equidimensional decompositons of polynomial ideals,
e.g.~\cite{decker1999}. We additionally present an algorithm for
minimizing a given Whitney stratification.

Our main objects of study will be algebraic varieties defined by systems of polynomial equations $f_1, \dots, f_r$ in the polynomial ring $\CC[x_1, \dots, x_n]$ over the field of complex numbers $\CC$:$$
\bV(f_1, \dots, f_r):=\left\lbrace (x_1, \dots, x_n)\in \CC^n\; |\; f_1(x)=\cdots =f_r(x)=0 \right\rbrace. 
$$ To demonstration of the utility, and arguably the necessity, of
Whitney stratifications when studying the geometry and topology of
singular varieties we consider a concrete example over the real numbers $\RR$ for illustrative purposes. Suppose we wish to study the curve in $\RR^2$  defined by the parametric polynomial \begin{equation}
   f_z(x,y)=(y-1)^2-(x-z)(x-1)^2 \label{eq:planarCubic}
\end{equation}
in variables $x,y$ with parameter $z$. For $z<1$ the real variety
$f_z=0$ is a (connected) nodal cubic, for $z>1$ the real variety
$f_z=0$ has two connected components and is smooth, and at $z=0$ the
curve is a cubic cusp, see \Cref{fig:curve}.

\begin{figure*}[h!]%
\captionsetup[subfigure]{justification=centering}
    \begin{subfigure}[t]{0.3\textwidth}
        \centering
         \includegraphics[width=0.8\linewidth]{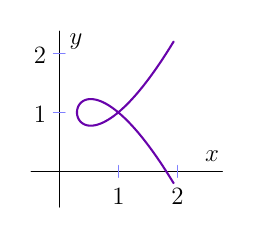}%
             \captionsetup{justification=raggedright,
singlelinecheck=false
}
        \caption{When $z<1$ we obtain a nodal cubic with one loop.}
        \label{fig:node}
    \end{subfigure}%
    \begin{subfigure}[t]{0.3\textwidth}
        \centering
         \includegraphics[width=0.8\linewidth]{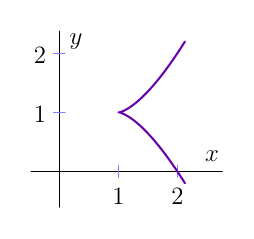}%
             \captionsetup{justification=raggedright,
singlelinecheck=false
}
        \caption{When $z=1$  we obtain a cusp cubic.}%
        \label{fig:cusp}
    \end{subfigure}%
    \begin{subfigure}[t]{0.3\textwidth}
        \centering
        \includegraphics[width=0.8\linewidth]{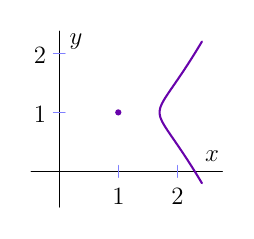}%
                    \captionsetup{justification=raggedright,
singlelinecheck=false
}
    \caption{When $z>1$, over the reals, we get a curve with two connected components. }
        \label{fig:smooth}
    \end{subfigure}%
                \captionsetup{justification=raggedright,
singlelinecheck=false
}
    \caption{Plots of the curve defined by \eqref{eq:planarCubic} for different parameter values $z$; the topology of the curve changes at $z=1$. While the curve in (\subref{fig:smooth}) is smooth and has two connected components (of different dimensions) in $\RR^2$ it is connected and singular, with singularity at $(1,1)$, in $\CC^2$.}
\label{fig:curve}
\end{figure*}
 
To understand the geometry underlying this change in topology at $z=1$ for the parametric curve $f_z$ consider the variety
$X=\bV(f)$ in $\RR^3$ defined by the same polynomial
$f(x,y,z)=(y-1)^2-(x-z)(x-1)^2$, plotted in 
\Cref{fig:Surface}. The surface $X$ is singular along the entire line
${\rm Sing}(X)=\bV(x-1, y-1)$ and the projection map $\pi:\RR^3\to \RR$
onto the last coordinate, $(x,y,z)\to z$, restricted to either the
manifold $X-\bV(x-1,y-1)$ or to the manifold $\bV(x-1,y-1)$ is a
submersion, however the behavior of the fibers of $\pi$ restricted to
$X-\bV(x-1,y-1)$ depends on whether $z$ is greater than or less than
$1$, while the behavior of the fibers of the restriction to
$\bV(x-1,y-1)$ is independent of $z$.
\begin{figure}[h!]\centering

\includegraphics[scale=0.25]{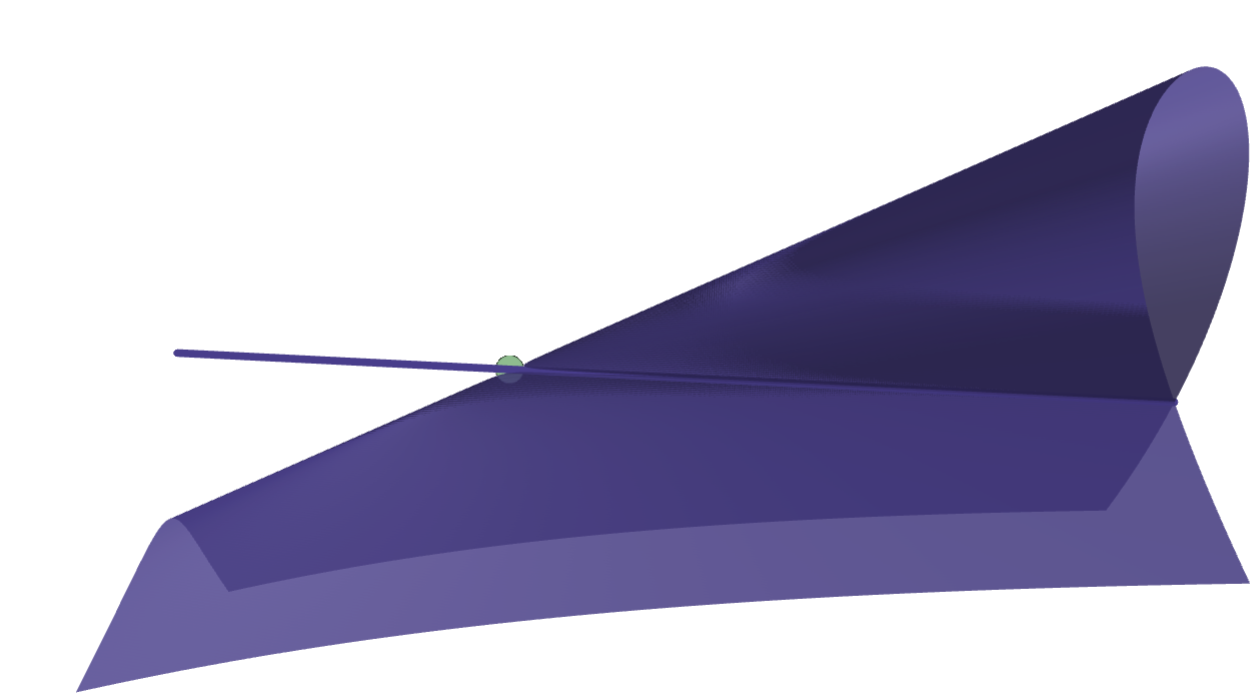}\caption{The surface $X=\bV\left((y-1)^2-(x-z)(x-1)^2 \right)$ in $\RR^3$. It's singular locus is $Y=\bV(x-1, y-1)$ and its Whitney stratification arises from the flag $\{(1,1,1)\} \subset Y \subset X$. The dimension 2 stratum is $X-Y$, the two dimension 1 strata are the connected components of $Y-\{(1,1,1)\}$, and the dimension 0 stratum is $(1,1,1)$. \label{fig:Surface}}
\end{figure}{\em Whitney's Condition (B)} (see Definition \ref{def:condb}) fails for the
pair of manifolds $X-\bV(x-1,y-1)$ and $\bV(x-1,y-1)$ at the point
$(1,1,1)$. The failure of condition (B) at this point captures the fact that a fundamental change in the local geometry of points on $\bV(x-1,y-1)$ relative to the surface $X$ occurs at the point
$(1,1,1)$. As it turns out, this in particular allows us to know exactly
when the topology of the fibers of the map $\pi$ change.

More precisely a classical result known as {\em Thom's Isotopy Lemma} \cite[Proposition 11.1]{Mather2012}
allows us to partition the codomain of (proper) maps into regions of
constant topology; Whitney's condition is a critical component of this
construction as illustrated by the example above \footnote{Note that
  in fact the map $\pi$ in the example is not proper, so Thom's Isopotpy
  lemma does not hold for the affine map, however we can see that the
  result still holds for this example by passing to the projective
  closure in $\PP^2\times \CC$, where the corresponding map is proper.}.

Effective algorithms to compute Whitney statifications of real and
complex algebraic varieties, and of algebraic maps between them, were
developed in \cite{hnFOCM,hnFOCMCorrection,hn2Real}. These algorithms
compute a Whitney stratification of a variety $X$ by way of computing
the set of associated primes of certain (scheme theoretic) fibers in the so-called {\em
  conormal space} associated to $X$ (see Definition \ref{def:conorm}). These fibers
are often much more complicated than their image, making the computation
of the required associated primes difficult even for relatively simple
varieties.
\begin{example} \label{ex:IntroBiggerEx}
  Consider the variety $X:=\bV(x_{1}^{7}-2\,x_{1}^{5}x_{4}+x_{1}^{3}x_{4}^{2}+x_{1}^{2}x_{2}^{2}-x_{1}x_{2}^{2}x_{3}+x_{2}^{3})\subset \CC^4$. The singular
  locus of $X$ is $X_{\rm Sing}=\bV(x_{2},\,x_{1}^{3}-x_{1}x_{4})$. Inside the singular locus Whitney's condition (B) fails to hold on the following subvarieties of dimension 1:
\begin{equation}
    \label{eq:WSex1}\bV\left(x_{4},\,x_{2},\,x_{1}\right),\:\bV\left(x_{2},\,x_{1},\,4\,x_{3}^{3
     }-27\,x_{4}^{2}\right),\:\bV\left(x_{2},\,x_{1}-x_{3},\,x_{3}^{2}-x_{4}\right). 
\end{equation} The three varieties listed in \eqref{eq:WSex1} intersect at the origin. The origin along with the varieties \eqref{eq:WSex1} and $X_{\rm Sing}$, define the Whitney stratification of $X$. As noted above this is computed via studying certain fibers of the conormal space. 
  The conormal space of $X$, ${\rm Con}(X)\subset \CC^4\times \PP^3$, is defined by an ideal generated by 22 polynomials in eight variables, most with between 8 and 10 terms each (for a total of 190 terms across all polynomials). The fiber of the
  singular locus in the conormal space of $X$, obtained by adding the ideal of ${\rm Con}(X)$ to the two equations defining $X_{\rm Sing}$, is specified by a non-radical ideal generated by 20 polynomials
  in eight variables. These polynomials contain between 1 and 16 terms each, for a total of 127 terms across all 20 polynomials.  For this latter ideal we were unable to compute the associated primes (though the minimal primes can be computed relatively quickly). 
\end{example}
In line with this observation, practical experience has shown that the
computation of these associated primes is precisely the main
bottleneck for these algorithms. We note that, theoretically speaking (to apply the known algebraic criterion for condition (B), see Theorem \ref{thm:conorm}),
for the ideals involved it is absolutely essential that we obtain both
the embedded and the isolated primary components. In particular,
algorithms which compute only the isolated primes of an ideal cannot
be used. In this paper we address this bottleneck using a modified
version of a well-known Gröbner basis algorithm for {\em
  equidimensional} decomposition (see e.g. \cite{decker1999}) which
allows us to compute the intersection of all required associated
primes of fibers in some conormal space with a single much simpler
Gröbner basis computation. While the worst case complexity of the
approach is unchanged this, nonetheless, yields quite substantial
speedups in practice on a wide variety of examples.  However, because
we may compute a larger variety, which contains the desired one, this
can yield non-minimal Whitney statifications even when the algorithm
of \cite{hnFOCM,hnFOCMCorrection,hn2Real} would yield a minimal
stratification. While this is not necessarily a problem, in some
applications it may be more desirable to obtain the minimal
stratification. To this end we also use classical results of Tessier
\cite[page 751--752]{FTpolar} based on the concept of {\em local polar
  varieties} to develop a second algorithm which ensures that the
minimal stratification will be produced (at least in the case of
complex varieties).

This paper is organized as follows: in \Cref{sec:background} we
briefly review the necessary background on Gröbner bases, Whitney
stratifications and local polar varieties. The main results, and the
resulting algorithm for computing the Whitney stratification of a
variety and for minimizing a given Whitney stratification, are
presented in \Cref{sec:new}. Finally in \Cref{sec:runtimes} we give
run time comparisons.

\subsection*{Acknowledgements}
{\footnotesize Martin Helmer  was supported by the Air Force Office of Scientific Research (AFOSR) under award
number FA9550-22-1-0462, managed by Dr.~Frederick Leve, and by the Royal Society under grant RSWF\textbackslash R2\textbackslash 242006, and would like to gratefully acknowledge this support. Rafael Mohr was supported by the DFG Sonderforschungbereich TRR 195, by the Forschungsinitiative Rheinland-Pfalz, by the ERC project 10000 DIGITS, by the FWO grants G0F5921N (Odysseus) and G023721N, and by the KU Leuven grant iBOF/23/064 and would like to gratefully acknowledge this support. The authors thank the anonymous reviewer for their valuable comments and suggestions.}

\section{Background}\label{sec:background}
In this section we briefly recall several important facts regarding Gröbner bases,
Whitney stratifications, and the effective version of Whitney's
condition given in \cite{hnFOCMCorrection}. These facts will be needed to
construct our new stratification algorithm in the sequel.

\subsection{Gröbner Bases}
\label{sec:groebner}

Let us briefly state the necessary definitions and properties of
Gröbner bases needed for our algorithm. Just for this section, let
$\field$ be any field and let $R:=\field[\xx]$ be the polynomial ring
over $\field$ in a finite set of variables $\xx$. Denote by
$\mon(\xx)$ the set of monomials in $\xx$.

\begin{definition}[Monomial Order]
  \label{def:monorder} A \emph{monomial order} $\mord$ on $\xx$ is a total
  order on $\mon(\xx)$ which:
  \begin{enumerate}
  \item extends the partial order on $\mon(\xx)$ given by divisibility, and
  \item is compatible with multiplication i.e.\ we have,
    \begin{align*}
      u \mord v \; \Rightarrow \; wu \mord wv \quad \forall u,v,w\in \mon(\xx).
    \end{align*}
  \end{enumerate}
\end{definition}

\noindent A monomial order on $\xx$ yields a notion of \emph{leading monomial}
 and \emph{leading coefficient} in $R$.

\begin{definition}[Leading Monomial]
  \label{def:lm} Let $\mord$ be a monomial order on $\mon(\xx)$. For a
  nonzero element $f\in R$ the \emph{leading monomial} of $f$ with respect to
  $\mord$, denoted $\lm_{\mord}(f)$, is the $\mord$-largest monomial
  in the support of $f$. The \emph{leading coefficient}, denoted
  $\lc_{\mord}(f)$, is the corresponding coefficient of
  $\lm_{\mord}(f)$ in $f$. For a finite set $F$ in $R$ we define
  $\lm_{\mord}(F) := \setof{\lm_{\mord}(f)}{f\in F}$.  For an ideal
  $I$ in $R$ we define the \emph{leading monomial ideal} of $I$ as
  $\lm_{\mord}(I):=\langle \lm_{\mord}(f)\;|\;f\in I\rangle$.
\end{definition}

\noindent We finally define the notion of Gröbner bases.

\begin{definition}[Gröbner Basis]
  \label{def:gb} A \emph{Gröbner basis} of an ideal $I\subset R$ with respect to a
  monomial order $\mord$ is a finite set $G\subset I$ such that
  $\langle \lm_{\mord}(G) \rangle = \lm_{\mord}(I)$. It is called {\em minimal} if
  for any $g\in G$, $\lm_{\mord}(g)$ is not divisible by any element in
  $\lm_{\mord}(G- \{g\})$
\end{definition}

\noindent Any Gröbner basis can be turned into a minimal one by
removing all elements whose leading monomials are divisible by another
leading monomial occuring in the Gröbner basis.

\noindent For our algorithm, we will need to compute Gröbner bases for
\emph{block orders}.

\begin{definition}[Block Order]
  Let $\xx$ and $\yy$ be two finite sets of variables. Write each
  monomial $u\in \mon(\xx \cup \yy)$ uniquely as a product
  $u = u_{\xx}u_{\yy}$ with $u_{\xx}\in \mon(\xx)$ and
  $u_{\yy}\in \mon(\yy)$. Fix a monomial order $\mord_1$ on
  $\mon(\xx)$ and a monomial order $\mord_2$ on $\mon(\yy)$. The
  corresponding \emph{block order eliminating $\xx$} is defined as
  follows: $u \mord v$ if and only if $u_{\xx} \mord_1 v_{\xx}$ or $u_{\xx}=v_{\xx}$
  and $u_{\yy} \mord_2 v_{\yy}$ for $u,v\in \mon(\xx\cup \yy)$.
\end{definition}

The key property of Gröbner bases with respect to block orders that we need is
the following (Lemmas 8.91 and 8.93 in \cite{becker1993}):

\begin{proposition}
  \label{prop:blo}
  Let $\xx$ and $\yy$ be any two finite sets of variables and let
  $\mord$ be any block order eliminating $\xx$. Let $G$ be a Gröbner basis
  of some ideal $I$ with respect to $\mord$. Then
  \begin{enumerate}
  \item $G$ is also a Gröbner basis of the ideal $I\field(\yy)[\xx]$.
  \item Let $H$ be the corresponding minimal Gröbner basis of $G$. Let
    $h\in \field[\yy]$ be the least common multiple of the leading
    coefficients of $H$, regarded as a subset of
    $\field(\yy)[\xx]$. Then
    \[I\field(\yy)[\xx] \cap \field[\xx,\yy] = \sat{I}{h}.\]
  \end{enumerate}
\end{proposition}

We will want to compute Gröbner bases of ideals of the form
$I\field(\yy)[\xx]$ when $\yy$ is chosen such that the map
$\field[\yy]\rightarrow \field[\xx,\yy]/I$ is injective. To this end, define:

\begin{definition}[Maximally Independent Subset]
  \label{def:mis} A {\em maximally independent subset} $\uu\subset \xx$ of an
  ideal $I\subset \field[\xx]$ is a set such that the map
  $\field[\uu]\rightarrow \field[\xx]/I$ is injective and such that the
  cardinality of $\uu$ is equal to $\dim(\field[\xx]/I)$.
\end{definition}

Computationally, we will determine maximally independent subsets from
any Gröbner basis of $I$ (Theorem 9.27 in \cite{becker1993}):

\begin{proposition}
  \label{prop:mis} Let $G$ be any Gröbner basis of an ideal
  $I\subset \field[\xx]$. Let $\uu\subset \xx$ be maximal such that no leading
  monomial of any $g\in G$ depends only on $\uu$. Then $\uu$ is a
  maximally independent subset.
\end{proposition}

\subsection{Whitney Stratifications}
Throughout the paper, let $\CC^n$ be the $n$-dimensional affine space
over the complex numbers $\CC$, $\PP^n$ be the $n$-dimensional
projective space over $\CC$ with dual projective space
$(\PP^{n})^{*}$, and let $\Gr(k,n)$ denote the projective variety of
$k$-dimensional vector spaces in $\CC^n$. We assume throughout the
rest of the paper that all varieties in $\CC^n$ are defined over
$\QQ$, i.e. cut out by $n$-variate polynomials with coefficients in
$\QQ$. For any variety $X\subset \CC^n$ we will denote by $I(X)$ the radical
ideal of all polynomials vanishing on $X$. We will frequently use the
notation $I_X$ to denote any ideal with $\bV(I_X) = X$. Note that such
an ideal $I_X$ is of course not uniquely defined.

Our goal is to compute a Whitney stratification of a singular variety
$X\subset \CC^n$, which is a certain flag on $X$ where the
successive differences between varieties in the flag are smooth and where these smooth pieces satisfy an additional regularity condition. This additional regularity condition, called {\em Whitney's Condition (B)} is designed to enforce that the local structure of each connected component of each smooth piece of the stratification is geometrically identical in a precise way. 

\begin{definition}[Whitney's Condition (B)]
  \label{def:condb}
  Let $Y\subset X$ be algebraic varieties in $\CC^n$ with $\dim(Y)< \dim(X)$ and with $X-Y$ smooth. Let $Y_{\rm reg}$ denote the smooth part of $Y$. The pair $(X,Y)$
  satisfies {\em Whitney's condition (B)} if for any $p\in Y_{\rm reg}$
  \begin{itemize}
  \item and any sequences $(x_i)\subset X$ and $(y_i) \subset Y_{\rm reg}$ converging to $p$,
  \item if the secant lines $\ell_i=[x_i,y_i]\in \PP^{n-1}$ converge to some $\ell \in \PP^{n-1}$,
  \item and if the tangent spaces $T_{x_i}(X-Y)\in \Gr(\dim(X), n)$ converge
    to some $T\in \Gr(\dim(X), n)$,
  \end{itemize}
  then $\ell \subset T$.
\end{definition}

This now allows to formally define
a Whitney stratification of a variety $X\subset \CC^n$, see e.g.~\cite{wallregular,whitney1965tangents}:

\begin{definition}[Whitney Stratification]
  A {\em Whitney stratification} of $X$ is a flag
  \[\emptyset = W_{-1} \subset W_0 \subset \dots \subset W_{k-1} \subset W_k = X\]
  where each difference $M_i:=W_i- W_{i-1}$ is a smooth variety. The
  connected components of each $M_i$ are called {\em strata} and each
  pair of strata must satisfy Whitney's condition (B). Note that the condition that $M_i$ is smooth in particular means that for each $W_i$ we must have that the singular locus  of $W_i$ is contained in $W_{i-1}$, i.e.~$(W_i)_{\rm Sing}\subset W_{i-1}$.

  A {\em minimal} Whitney stratification of $X$ is a Whitney
  stratification such that, after removing any stratum, the resulting
  stratification fails to be a Whitney stratification. The existence of a unique minimal stratification follows from results of \cite{tessier1981varieties}, see also \cite{FTpolar}. 
\end{definition}

In \cite{hnFOCM} an algebraic criterion based on primary decomposition
of polynomial ideals is given that allows to check computationally if
a given pair $X,Y\subset \CC^n$ satisfies condition (B), and as such, to
compute a Whitney stratification of a given variety by successively
computing singular loci and then checking condition (B) for all pairs
of resulting varieties. The central contribution of this present work
is to still apply the same algebraic criterion but to give a simple
way in which having to explicitly compute primary decompositions can
nonetheless be avoided, see \Cref{sec:newcrit}.

Let us now fix an {\em equidimensional}, affine variety $X\subset \CC^n$, where
equidimensional means that all irreducible components of $X$ have the
same dimension. We now introduce the aforementioned algebraic
criterion, which will use the notion of a {\em conormal space}.
\begin{definition}[Conormal Space]
  \label{def:conorm}
  Denote by $\reg{X}$ the set of regular points of $X$. The {\em
    conormal space} $\Con(X)\subset \CC^n\times (\PP^{n-1})^{*}$ of
  $X$ is the Zariski closure of the set
  \[\setof{(p,\zeta)}{p\in \reg{X}\text{ and }T_p\reg{X}\subset \zeta^\perp }.\]
  The canonical projection $\kappa_X:\Con(X)\rightarrow X$ is called the {\em
    conormal map}.
\end{definition}

This notion was used in \cite{hnFOCM} to give the aforementioned
algebraic criterion for condition (B), which we restate below\footnote{We note that in \cite{hnFOCM} this theorem is stated and proved (in the correction \cite{hnFOCMCorrection} to \cite{hnFOCM}) for the case where $X$ (and hence $Y$) are projective varieties. However the restriction to the projective case is not necessary and the same theorem holds, with exactly the same proof in the affine case (specifically the proof of Theorem 2.1 in \cite{hnFOCMCorrection}, which is the corrected version of Theorem 4.3 in \cite{hnFOCM}) . Hence we do not redo the proof here and simply re-phrase the final result to the affine setting.}.

\begin{theorem}[Theorem 4.3 in \cite{hnFOCM}]
  \label{thm:conorm}
  Let $\emptyset \neq Y\subset \sing{X}$ be equidimensional and let
  $I_Y$ be any ideal with $\bV(I_Y) = Y$. Let
  \[I_Y + I(\Con(X)) = \bigcap_{i\in I} Q_i\] be a primary decomposition. Let
  \[J:=\setof{i\in I}{\dim(\kappa_X(\bV(Q_i))) < \dim (Y)}.\]
  Further, let
  \[A:= \left[\bigcup_{j\in J}\kappa_X(\bV(Q_i)) \right]\cup \sing{Y}.\] Then the pair
  $(\reg{X}, Y- A)$ satisfies condition (B).
\end{theorem}

\subsection{Local Polar Varieties and Minimal Whitney Stratifications}

While our new algebraic criterion for Whitney's condition (B) avoids
the computation of primary decompositions, and thus speeds up the
stratification process, it may fail to produce a minimal Whitney
stratification even when the algorithm in \cite{hnFOCM} does. In order
to rectify this, we will use the concept of {\em local polar
  varieties} which we introduce here.

Let $X\subset \CC^n$ be a variety of dimension $d$ and consider the
conormal space $\Con(X)\subset \CC^n \times \PP^{n-1}$ with associated conormal
map $\kappa_X: \Con(X)\to X$. For fixed $i$ with $1\leq i \leq d-1$, consider a
dimension $d+i-1$ linear space in $\CC^n$ with dual
$L_i\subset (\PP^{n-1})^{*}$. We use the convention $L_0=\CC^n\times (\PP^{n-1})^{*}$.
\begin{definition}[Polar Variety]
  A {\em polar variety} is a variety of the form
$$ \delta_i(X):=\kappa_X(\Con(X) \cap L_i).$$
\end{definition}

Consider now a point $y\in X$. Picking our linear spaces as general
linear spaces through $y$ with dual $\tilde{L}_i\subset \PP^{n-1}$ we define
\begin{definition}[Local Polar Variety]
  A {\em codimension $i$ local polar variety through $y$} is a variety
  of the form
  \[\delta_i(X,y) := \kappa_X(\Con(X) \cap \tilde{L}_i).\]
\end{definition}

Since the conormal space is reduced (as a scheme), i.e.~is a variety \cite[Proposition 2.9]{FTpolar},  and since the $L_i$ (and the $\tilde{L}_i$) are general it follows that the local polar varieties are also reduced \cite[Remark~3.10 (b)]{FTpolar} and are hence are each defined by a radical ideal. 

If our linear space $\tilde{L}_i$ are chosen sufficiently generic,
then, using a dimension count, some of the local polar varities will
contain $y$ and some will not (see \cite[Remark 3.1]{FTpolar}). We
will want to compute the {\em multiplicity} of points $y$ in local
polar varieties $\delta_i(X,y)$. Let us define the notion of multiplicity
used here:

\begin{definition}[Multiplicity]
  Let $Z\subset \CC^n$ be a subvariety and let $z\in \CC^n$ and write
  $\mathfrak{m}_z:=I(z)$. If $z\in Z$, the {\em multiplicity} of $Z$ at
  $z$, denoted $m_z(Z)$, is defined as the Hilbert-Samuel multiplicity
  (see e.g. Chapter 12 in \cite{eisenbud1995}) of the local ring
  $(\CC[\xx]/I(Z))_{\mathfrak{m}_z}$ at $\mathfrak{m}_z$. If $z\notin Z$ we define $m_z(Z):=0$.
\end{definition}

For a subvariety $Y\subset X$ and a given point
$y\in Y$, computing multiplicities of the form $m_y(\delta_i(X,y))$ can be
used to check if Whitney's condition (B) holds for $X$ and $Y$ at $y$.
More precisely (see p. 69 and Proposition 3.6 in \cite{FTpolar}):
\begin{theorem}
  \label{thm:multseq}
  Let $Y\subset X$ be a subvariety and $y\in Y$. Then the sequence
  \[m_{\bullet}(X,x) :=
    (m_x(X),m_x(\delta_1(X,x)),\dots,m_x(\delta_{\dim(X)-1}(X,x)))\] is
  independent of the linear subspaces chosen to construct the local
  polar varieties if they are sufficiently general. In addition,
  Whitney's condition (B) is satisfied at $y$ for $X$ and $Y$ if the
  sequence takes the same value for every $x$ in a euclidean
  neighborhood of $y$ in $Y$.
\end{theorem}
Making this theorem computationally profitable will require us to
compute sequences $m_{\bullet}(X,y)$ where $y$ lies in a subvariety
$Y\subset X$. We will show how to do this in \Cref{sec:min}.

\section{New Whitney Stratification Algorithms}\label{sec:new}
In this section we present the main results of this work, namely a new algebraic condition to identify subvarieties where condition (B) fails, new algorithms to compute a Whitney stratification of an algebraic variety and an algorithm to produce the unique minimal Whitney stratification for complex varieties along with theoretical results which guarantee the correctness of our algorithms. Our new formulation of Condition B is given in \Cref{sec:newcrit}, the resulting stratification is given in \Cref{sec:alg}, and finally an algorithm to coarsen a given Whitney stratification to the unique minimal stratification is given in \Cref{sec:min}.

\subsection{A New Algebraic Criterion for Checking Whitney's Condition}\label{sec:newcrit}

Throughout this section we denote by $\CC[\xx]$ the underlying
polynomial ring of $\CC^n$. We will study the problem of effectively computing a Whitney stratification of a complex affine algebraic variety $X\subset \CC^n$. We first show how to, in the setting of Theorem \ref{thm:conorm}, compute a
superset of $A$ without having to compute a primary decomposition of
$I_Y+ I(\Con(X))$. For this we denote by $\CC[\xx,\yy]$ the underlying
polynomial ring of $\CC^n\times \PP^{n-1}$, where $\PP^{n-1}$ denotes the
$(n-1)$-dimensional projective space over $\CC$.

First recall the following elementary proposition from commutative
algebra:

\begin{proposition}
  \label{prop:dim} Let $P\subset \CC[\xx]$ be a prime ideal and let
  $\uu\subset \xx$ be of cardinality $d$.  If $P\cap \CC[\uu] = 0$, then
  $\dim(\CC[x]/P) \geq d$.
\end{proposition}
\begin{proof}
  This follows from the definition of Krull dimension of an ideal: Any
  maximal chain of prime ideals in $\CC[\uu]$ extends to a chain of
  prime ideals in $\CC[\xx]/P$ since $\CC[\uu]\rightarrow \CC[\xx]/P$ is
  injective.
\end{proof}

 Using \Cref{thm:conorm} we obtain an
alternative algebraic criterion which will also allow us to construct
a Whitney stratification. When reading the result below keep in mind that $\xx, \uu,$ and $\yy$ (and hence $\xx-\uu$) denote sets of variables. 
\begin{theorem}
  \label{cor:whitpnts} Let $X\subset \CC^n$ be an algebraic variety and let $Y$ be any (not necessarily equidimensional) subvariety of the
singular locus $\sing{X}$ of $X$.
  Choose any maximal independent subset of variables $\uu\subset \xx$ of any
  defining ideal $I_Y$ of $Y$. Let $J:=I(\Con(X)) + I_Y$. Let $G$ be a
  minimal Gröbner basis of ${J\CC(\uu)[\xx- \uu,\yy]}$ with respect to any
  monomial order $\mord$ with $G\subset \CC[\xx,\yy]$. Let
  \[h = \lcm\setof{\lc_{\mord}(g)}{g\in G}\in \CC[\uu].\]
  Then
  \begin{enumerate}
  \item \label{en:1} $Y - \bV(h)$ is equidimensional of dimension $\dim(Y)$.
   \item \label{en:2} The pair $(X,Y- (\bV(h)\cup \sing{Y}))$ satisfies
    Whitney's condition (B).
  \end{enumerate}
\end{theorem}
\begin{proof}
  Throughout this proof, we denote by $\ass(I)$ the set of associated
  primes of any ideal $I$.
  
  ({\em Proof of \ref{en:1}}) By Proposition \ref{prop:blo}, we have
  \begin{equation}
    \label{eq:idl}
    \sat{J}{h} = J\CC(\uu)[\xx- \uu, \yy] \cap \CC[\xx,\yy].
  \end{equation}
  The ring $\CC(\uu)[\xx- \uu, \yy]$ is the localization of
  $\CC[\xx, \yy]$ at the multiplicative
  set~$\CC[\uu]- \{0\}$. Hence, using Theorem 3.1 in \cite{eisenbud1995},
  \[\ass(\sat{J}{h}) = \setof{P\in \ass(J)}{P \cap \CC[\uu] = 0}.\]
  Note that $\sqrt{J}\cap \CC[\xx] = I(Y)$ and hence, since $h\in \CC[\uu]$,
  $\sqrt{\sat{J}{h}}\cap \CC[\xx] = \sqrt{\sat{I_Y}{h}}$. Hence, for $P$ a
  minimal prime over $\sat{I_Y}{h}$, there is a minimal prime $Q$ over
  $\sat{J}{h}$ with $Q\cap \CC[\xx] = P$. Then we have
  \[0 = Q \cap \CC[\uu] = P \cap \CC[\uu]\] and therefore
  $\dim (\bV(P)) \geq \dim (Y)$ by \Cref{prop:dim}. On the other hand, $P$ is also
  minimal over $I_Y$.  Therefore $\dim(\bV(P)) \leq \dim(Y)$ and finally
  $\dim(\bV(P)) = \dim(Y)$. This shows that $Y- \bV(h)$ is equidimensional of
  dimension equal to $\dim(Y)$.

  ({\em Proof of \ref{en:2}}): Let $P\in \ass(J)$ such that
  $\dim(\kappa_X(\bV(P))) < \dim(Y)$. This implies
  ${P \cap \CC[\uu] \neq 0}$ again by Proposition \ref{prop:dim}, therefore
  $P\notin \ass(J\CC(\uu)[\xx- \uu, \yy])$ and therefore also
  ${P\notin \ass(\sat{J}{h})}$. Hence necessarily $h\in P$ or, since
  $h\in \CC[\uu]$, $\kappa_X(\bV(P)) \subset \bV(h)$. If we then define
  \[A := \bigcup_{\substack{P\in \ass(J) \\\dim(\kappa_X(\bV(P))<
        \dim(Y)}}\kappa_X(\bV(P))\] then we have $A \subset \bV(h)$. By \Cref{thm:conorm},
  the pair $(X, Y- (A\cup \sing{Y}))$ satisfies Whitney's condition (B) and
  since $Y- \bV(h) \subset Y- A$, so does $(X, Y- (\bV(h)\cup \sing{Y}))$.
\end{proof}

\subsection{The Whitney Stratification Algorithm}
\label{sec:alg}

To start, let us briefly recall how to compute an ideal defining the
conormal space of a given $X\subset \CC^n$ (see also Section 4.1 in
\cite{hnFOCM}). Suppose that $X = \bV(f_1,\dots,f_r)$ for certain
$f_1,\dots,f_r\in \CC[\xx]$. Let $F:=(f_1,\dots,f_r)$. We then build the
augmented Jacobian matrix
\[\jac_{\yy}(F):=
  \begin{bmatrix}
    y_0 & y_1 &\dots &y_{n-1}\\
    \partial_1f_1 &\partial_2f_1 &\dots &\partial_nf_n\\
    \vdots & & &\vdots\\
    \partial_1f_r &\partial_2f_r &\dots &\partial_nf_r\\
  \end{bmatrix}
\]
with new variables $\yy:=\{y_0,\dots,y_{n-1}\}$. Here, $\partial_i$ denotes
the partial derivative by the $i$th variable in $\xx$. Denote by
$\jac(F)$ the usual Jacobian matrix of $F$. Now (see again
\cite{hnFOCM}) we have the following fact. 
\begin{proposition}
  \label{prop:computeconorm}
  With the notation above let $c:=\codim(X)$, let $M$ be the set of $(c\times c)$-minors of $\jac(F)$,
  and let $M_{\yy}$ be the set of $(c+1)\times (c+1)$-minors of
  $\jac_{\yy}(F)$. Then
  \[\Con(X) = \bV\sat{(\langle F\rangle + \langle M_{\yy}\rangle)}{M}. \]
\end{proposition}

\begin{remark}
  Recall that the necessary saturation to compute conormal spaces can
  be performed using Gröbner basis computations with block orders and
  {\em Rabinowitsch's trick}, see e.g. Exercise 15.41 in
  \cite{eisenbud1995}.
\end{remark}

Now we can give the following algorithm subroutine {\bf WhitPoints}, which given
any pair of varieties $X$ and $Y$ with $Y\subset \sing{X}$, returns a
polynomial $h$ as in \Cref{cor:whitpnts}. This subroutine is the
centerpiece of our algorithm for computing Whitney stratifications. As
previously mentioned, we avoid the computation of associated primes in
the conormal space of $X$ which an immediate application of
\Cref{thm:conorm} would require.

\medskip
\begin{center}
  \begin{tabular}{|r|l|}
    \hline
    ~ & {\bf WhitPoints}$(X,Y)$ \\
    \hline
    ~&{\bf Input:} An equidimensional affine variety $X$, any closed $Y\subset \sing{X}$.\\
    ~&{\bf Output:} An element $h\in \CC[\uu]$ as in Theorem \ref{cor:whitpnts}.\\
    \hline
    1 & $I_X\gets $any ideal defining $X$, $I_Y\gets $any ideal defining $Y$.\\
    2 & Let $\uu$ be any maximally independent subset of variables of $I_Y$.\\
    3 & Let $\prec$ be any monomial ordering eliminating $(\xx\cup \yy)- \uu$.\\
    4 & {\bf Set} $G$ to be a $\prec$-Gröbner basis of $I(\Con(X)) + I_Y$.\\
    5 & Minimize $G$ over $\CC(\uu)[\xx- \uu, \yy]$.\\
    6 & {\bf Return} $h:=\lcm\setof{\lc(g)}{g\in G}$.\\
    \hline
  \end{tabular}
\end{center}
\medskip

\begin{remark}
  In line 2, we use a Gröbner basis of $I_Y$ and Proposition \ref{prop:mis} to compute
  the desired maximally independent subset of variables of $I_Y$.
\end{remark}

\begin{theorem}
  \label{thm:whitpntscor}
  For a given pair of varieties $X$ and $Y\subset \sing{X}$, the output polynomial $h$
  of $\mthcall{WhitPoints}(X,Y)$ is such that the pair $(X,Y- (\bV(h)\cup \sing{Y}))$ satisfies
  Whitney's condition (B).
\end{theorem}
\begin{proof}
   Note that, thanks to Proposition \ref{prop:blo}, a Gröbner basis of
  an ideal $I\subset \CC[\xx]$ with respect to a monomial order eliminating
  $\xx- \uu$, where $\uu\subset \xx$, is also a Gröbner basis of
  $I\CC(\uu)[\xx- \uu]$. The stated result then follows immediately
  from \Cref{cor:whitpnts}.
\end{proof}

Finally, we can use this subroutine to give the new algorithm {\bf
  Whitney} below to compute a Whitney stratification of an affine
equidimensional variety $X$. In the this algorithm, each occurring
variety $Z$ is represented as a union of its $\QQ$-irreducible components
$Z = \bigcup_{i=1}^rZ_i$. Let us define three subroutines used by this algorithm:
\begin{itemize}
\item We define {\bf Components}($Z$) to return the set
  $\{Z_1,\dots,Z_r\}$.
\item For two varieties $Z, W$ we define the routine {\bf Add}($Z, W$)
  to return $Z \cup W$ represented again by its $\QQ$-irreducible
  components.
\item For a flag $\WS$ and a an irreducible variety $Z$ we then define
  $\mthcall{Update}(\WS,Z)$ to change $W_{\dim(Z)}$ to
  $\mthcall{Add}(W_{\dim(Z)}, Z)$.
\end{itemize}

\begin{remark}
  Note that, as \Cref{cor:whitpnts} does not rely on $Y$ being
  irreducible, it is not strictly needed for the correctness of our
  algorithm to represent the occuring varieties by their
  $\QQ$-irreducible components. We chose to present the algorithm like
  this here for ease of reading. It would not be difficult to adapt
  Algorithm 5 in \cite{decker1999} to our situation, then we would
  have to perform only {\em equidimensional decomposition} instead of
  finding the minimal primes over a given ideal which is generally
  much easier.

  However, as remarked in the introduction, the key contribution of
  this work is to avoid the computation of associated primes {\em in
    conormal spaces}. Even with our new method, we observed that the
  required computations in conormal spaces are still the bottleneck of
  our algorithm, compared to representing varieties by their
  irreducible components as above.
\end{remark}

\begin{center}
  \begin{tabular}{|r|l|}
    \hline
    ~ & {\bf Whitney}$(X)$ \\
    \hline
    ~&{\bf Input:} An equidimensional affine variety $X$.\\
    ~&{\bf Output:} A Whitney stratification of $X$.\\
    \hline
    1 & $d\gets \dim(X)$\\
    2 & $W_0\gets \emptyset$, ..., $W_{d-1}\gets \emptyset$, $W_d\gets X$\\
    3 & {\bf For} $i$ from $d$ to $0$\\
    4 & \spc $Z_1,\dots,Z_r\gets \mthcall{Components}(W_i)$\\
    5 & \spc {\bf For} $j,k$ from $1$ to $r$ with $j<k$\\
    6 & \spc \spc $\mthcall{Update}(\WS, Z_j\cap Z_k)$\\
    7 & \spc {\bf For} $Z$ in $\mthcall{Components}(W_i)$\\
    8 & \spc \spc $\mthcall{Update}(\WS, \sing{Z})$\\
    9 & \spc \spc {\bf For} $j$ from $d-1$ to $0$\\
    10 & \spc \spc \spc {\bf For} $Y$ in $\mthcall{Components}(W_j)$ if $Y\subset Z$\\
    11 & \spc \spc \spc \spc $h\gets \mthcall{WhitPoints}(Z,Y)$\\
    12 & \spc \spc \spc \spc $Y' \gets Y\cap \bV(h)$\\
    13 & \spc \spc \spc \spc $\mthcall{Update}(\WS,Y')$\\
    14 & {\bf Return} $\WS$\\
    \hline
  \end{tabular}
\end{center}

\begin{theorem}
  For a given equidimensional affine variety $X\subset \CC^n$,
  $\mthcall{Whitney}(X)$ terminates and outputs a Whitney
  stratification of $X$.
\end{theorem}
\begin{proof}
  The termination of the algorithm is clear. Note that the singular
  locus of any variety consists of the the union of the singular loci
  of its $\QQ$-irreducible components together with the intersections
  of all its $\QQ$-irreducible components. Thus lines 4 and 6 in
  $\mthcall{Whitney}$ guarantee the smoothness of all strata of the
  output of $\mthcall{Whitney}(X)$. Further, the fact that Whitney's
  condition (B) is satisfied for all pair of the strata of the output
  of $\mthcall{Whitney}(X)$ follows immediately from
  \Cref{thm:whitpntscor}.
\end{proof}

\begin{example}
  We illustrate our algorithm $\mthcall{Whitney}$ by running it on the
  variety
  \[X := \bV(z(zx-y^2+tx^3))\subset \CC^4.\] This is Example 4.1 in
  \cite{parusinski2025}.

  In the first iteration of the outermost for-loop starting on line 3
  in $\mthcall{Whitney}(X)$ our algorithm decomposes $X$ into its
  irreducible components $Y_1:=\bV(z)$ and
  $Y_2:=\bV(zx-y^2+tx^3)$. The irreducible intersection
  $Z:=\bV(z,tx^3-y^2)$ of $Y_1$ and $Y_2$ is added to $W_2$ in line 6
  of $\mthcall{Whitney}$. This intersection coincides with the
  singular locus of $X$, hence the next interesting step in our
  algorithm happens in line 11 when we call
  $\mthcall{WhitPoints}(Y_1,Z)$. This call finds $h = x^3$ and hence
  the variety $Z\cap \bV(x) = \bV(x,y,z)$ is added to $W_1$ in line
  13. Note that $\bV(x,y,z)$ actually coincides with the singular
  locus of $Z$.

  In the next iteration of the outermost for-loop in
  $\mthcall{Whitney}(X)$, we check Whitney's condition (B) of the
  pairs consisting of $Z$ (which has dimension $2$) and any lower
  dimensional variety in our intermediate stratification $\WS$, the
  only such lower dimensional variety is $\bV(x,y,z)$. This turns out
  to be the last step where our stratification is updated, namely when
  we call $\mthcall{WhitPoints}(Z, \bV(x,y,z))$ which returns $h =
  t$. Hence, the variety $\bV(x,y,z,t)$ is added to $\WS$. The output
  Whitney stratification finally consists of
  \begin{align*}
    &W_3 = X = \bV(z) \cup \bV(zx-y^2+tx^3)\\
    &W_2 = Z = \bV(z,tx^3-y^2)\\
    &W_1 = \bV(x,y,z)\\
    &W_0 = \bV(x,y,z,t).
  \end{align*}
  
\end{example}

\begin{remark}
While we have focused on the case of varieties over $\CC$ in this note, applying the results of \cite{hn2Real},  it follows that the algorithm {\bf Whitney} presented above would also give a valid stratification for real algebraic varieties. In particular, this follows by \cite[Theorem 3.3]{hn2Real} and the fact that ideal addition and Gr\"obner basis computation leave the coefficient field of the polynomials unchanged. Hence a Whitney stratification of a real algebraic variety may also be computed using these techniques. Note however, that the resulting stratification may fail to be minimal (see \cite{hn2Real} for more discussions) and additionally the algorithm in Section \ref{sec:min} below to minimize a given Whitney stratification {\bf does not} necessarily return the minimal stratification of a real variety.     
\end{remark}

\subsection{The Minimization Algorithm}
\label{sec:min}

Our first goal here is to show how to compute, for a given
$X\subset \CC^n$ and a random point $y$ in a subvariety $Y\subset X$, the sequence
$m_{\bullet}(X,y)$ without being given $y$ explicitly.  Next, we will show
that the sequence $m_{\bullet}(X,\bullet)$ is constant on a Zariski-open subset of
$Y$ if $Y$ is equidimensional.  These two things combined allow us to
compute, probabilistically, the ``generic'' multiplicities
$m_{\bullet}(X,y)$, $y\in Y$.  Combining this with \Cref{thm:multseq} we
obtain a procedure to produce the unique minimal Whitney
stratification of a complex algebraic variety given a Whitney
stratification computed by \mthcall{Whitney}.

For convenience we introduce the following notation: Given a sequence
of polynomials $F:=(f_1,\dots,f_n)$ with $f_i\in \CC[\xx]$ we denote by
$\phi_F$ the map
\begin{align*}
  \phi_F:\;&\CC^n\rightarrow \CC^n\\
  &p:= (p_1,\dots,p_n)\mapsto (f_1(p),\dots,f_n(p)).
\end{align*}
Recall that $\phi_F$ is said to be {\em regular} at a point $p\in \CC^n$ if
the Jacobian determinant of $F$ does not vanish at $p$.

Suppose for the moment that $Y$ is of codimension $c$ with
$Y = \bV(f_1,\dots,f_c)$ and that there are degree one polynomials
$\ell_{c+1},\dots,\ell_n$ such that, letting
$F:=(f_1,\dots,f_c,\ell_{c+1},\dots,\ell_n)$, the map $\phi_F$ is regular at
every point $y\in Y$.

Let $\tilde{X}$ and $\tilde{Y}$ be the closures of the images under
$\phi_F$ of $X$ and $Y$ respectively. Denote by $z_1,\dots,z_n$ the
coordinates on the codomain of $\phi_F$. Note that the origin
$0\in \CC^n$ lies in $\tilde{Y}$. Now we can compute the polar
varieties $\delta_i(\tilde{X},0) = \delta_i(\tilde{X})$ and their multiplicities
at $0$ by choosing linear spaces through the origin. This gives us the
correct multiplicities because of the following fact. 
\begin{proposition}
  \label{prop:mult}
  Using the notation above, for any $y\in \phi^{-1}(0)$ we have
  \[m_{\bullet}(X,y) = m_{\bullet}(\tilde{X},0).\]
\end{proposition}
\begin{proof}
  As $\phi_F$ was assumed to be regular at every $y\in Y$ it defines a local
  isomorphism near any $y\in Y$ by the inverse function theorem. This
  means that for any $y\in Y$, $\phi_F$ induces an isomorphism of completions
  \[\phi^{*}_{y,F} : \widehat{\cO}_{\CC^n,\phi(y)}\rightarrow \widehat{\cO}_{\CC^n,y}. \]
  Now, since the Hilbert-Samuel multiplicity of a local ring
  $(R,\mathfrak{m})$ depends only on $\gr_{\mathfrak{m}}(R)$ and since
  $\gr_{\mathfrak{m}}(R) = \gr_{\mathfrak{m}}(\widehat{R})$, the result follows.
\end{proof}

We note that the proposition above is not surprising since it is known that polar multiplicities are analytic invariants  of a germ of a reduced
equidimensional complex analytic space $(Z,0)\subset (\CC^n,0)$ and are thus preserved by local isomorphism; see e.g.~\cite[Proposition 3.13]{FTpolar} and the surrounding discussion and references.

Next suppose that $Y\subset X$ is equidimensional of codimension $c$ but no
longer necessarily cut out by $f_1, \dots,f_c$ as above. Then we can reduce
to the case considered above as follows:

\begin{proposition}
  \label{prop:dense}
  Let $G$ be any generating set of the ideal $I(Y)$. For $c$ general
  linear combinations $f_1,\dots,f_c$ of the elements in $G$ and
  general degree one polynomials $\ell_1,\dots,\ell_{n-c}$ there is a
  Zariski-dense open subset $U$ of $Y$ such that the map $\phi_F$ is regular at every $y\in U$, where
  $F:=(f_1,\dots,f_c,\ell_1,\dots,\ell_{n-c})$.
\end{proposition}
\begin{proof}
  For general $f_1,\dots,f_c$ as in the statement of the proposition,
  $f_1,\dots,f_c$ defines a maximal regular sequence in $I(Y)$. Hence
  there exists some polynomial $h$ such that
  $hI(Y)\subset \langle f_1,\dots,f_c\rangle$. Hence the open subset
  $U_1:=Y - \bV(h)$ is Zariski dense in $Y$ and, locally at every point
  $p\in U_1$, the ideal $I(Y)$ is generated by $f_1,\dots,f_c$. To prove
  the proposition it therefore suffices to consider the case where
  $I(Y) = \langle f_1,\dots,f_c\rangle$.

  By applying Bertini's theorem stated as in Theorem A.8.6 in
  \cite{sommeseNumericalSolutionSystems2005} to every irreducible
  component of $Y$, we see that for for generic fixed degree one
  polynomials $\ell_1,\dots,\ell_{n-c}$ there is a Zariski-open subset
  $U_2\subset \CC^{n-c}$ such that for
  $\mathbf{a}:=(a_1,\dots,a_{n-c})\in U_2$ the variety
  \[Z_{\mathbf{a}}:=\bV(f_1,\dots,f_c,\ell_1-a_1,\dots,\ell_{n-c}-a_{n-c})\] is
  zero-dimensional, does not meet the singular locus of $Y$ and is
  smooth. After potentially shrinking $U_2$, the variety $Z_{\mathbf{a}}$
  intersects every irreducible component of $Y$ for every
  $\mathbf{a}\in U_2$. Let
  $F:=(f_1,\dots,f_c,\ell_1,\dots,\ell_{n-c})$ and let $U$ be the preimage
  of $U_2$ under $\phi_F$ in $Y$. Note that $U$ is dense in $Y$. As
  $\langle f_1,\dots,f_c\rangle$ is a radical ideal, the smoothness condition on
  $Z_{\mathbf{a}}$ means precisely that the Jacobian determinant of $F$ does not
  vanish at any point $p\in U$, proving the statement.
\end{proof}
From this we obtain the following corollary. 
\begin{corollary}
  \label{cor:multcomp}
  Let $m_{\bullet}(\tilde{X},0)$ be the multiplicity sequence constructed as
  at the beginning of this section, with $F$ chosen as in Proposition
  \ref{prop:dense}. Then $m_{\bullet}(\tilde{X},0) = m_{\bullet}(X,y)$ for every
  $y\in \phi^{-1}(0)$.
\end{corollary}
\begin{proof}
  Let $U\subset \CC^n$ be as in Proposition \ref{prop:dense}. By the same proposition,
  with probability $1$, the intersection
  $Y\cap \bV(\ell_{c+1},\dots,\ell_n)\cap U \subset \phi^{-1}(0)$ is not empty.  For any
  choice of $y\in Y\cap \bV(\ell_{c+1},\dots,\ell_n)\cap U$ we then have
  $m_{\bullet}(\tilde{X},0) = m_{\bullet}(X,y)$ by Proposition \ref{prop:mult}.
\end{proof}

Now, given $X$ and $Y$ we can sketch an algorithm that computes
the sequence $m_{\bullet}(X,y)$ for a random point $y\in Y$.

\medskip
\begin{center}
  \begin{tabular}{|r|l|}
    \hline
    ~ & {\bf Mult}$(X,Y)$ \\
    \hline
    ~&{\bf Input:} An affine variety $X$, any closed $Y\subset X$.\\
    ~&{\bf Output:} The sequence $m_{\bullet}(X,y)$ for a randomly chosen $y\in Y$.\\
    \hline
    1 & $I_Y\gets \langle g_1,\dots,g_r\rangle$, the radical ideal defining $Y$.\\
    2 & $c\gets \codim(Y)$.\\
    3 & Choose $f_1,\dots,f_c$ as random linear combinations of the $g_i$.\\
    4 & Choose $\ell_{c+1},\dots,\ell_n$ as random degree one polynomials.\\
    5 & Compute ideals defining $\phi(X)=\delta_0(\phi_F(X),0),\dots,\delta_{\dim(X)-1}(\phi_F(X),0)$\\
    6 & Using these ideals, {\bf Return} $(m_0(\delta_0(\phi_F(X),0)),\dots, m_0(\delta_{\dim(X)-1}(\phi_F(X),0)))$\\
    \hline
  \end{tabular}
\end{center}
\medskip

\begin{remark}
  In line 5, an ideal defining $\phi(X)$ is again computed using Gröbner
  basis computations with block orders, see e.g. Proposition 15.30 in
  \cite{eisenbud1995}. Using the notation of
  Proposition \ref{prop:computeconorm}, ideals defining the necessary local polar
  varieties can then be computed as ideals of the form
  \[\sat{I(\phi(X))+\langle M_{\yy}\rangle + \tilde{L}}{M}\cap \CC[\xx]\]
  where $M_{\yy}$ are the suitable minors of the augmented Jacobian
  matrix associated to generators of $I(\phi(X))$, $M$ are the suitable
  minors of the Jacobian matrix of generators of $I(\phi(X))$ and
  $\tilde{L}$ is a collection of degree one forms, defining a linear
  space of suitable dimension in $\PP^{n-1}$.

  In line 6, to compute the required multiplicities we use degree computations via Gr\"obner basis as described in \cite{HH19} and implemented in the \texttt{SegreClasses} Macaulay2 package \cite{M2}, see in particular \cite[Theorem 5.3]{HH19}. Alternatively one can use
  standard basis computations, see e.g. \cite{sayrafi2017} for
  details.

\end{remark}

Our next goal is the following: Suppose again that $Y\subset X$ is
equidimensional.
Then we want to show that the sequence
$m_{\bullet}(X,\bullet)$ is constant on a Zariski-dense open subset of
$Y_{\rm reg}$ and, since $Y_{\rm reg }$ is dense in $Y$, hence also on a Zariski-dense open subset of $Y$. This is concluded as follows from \Cref{thm:multseq}:

\begin{corollary}\label{cor:polarmultDenseSet}
  Given a variety $X$ and equidimensional $Y\subset X$, there exists a
  Zariski-dense open subset $U\subset Y_{\rm reg}$ such that
  $m_{\bullet}(X, \bullet)$ is constant on $U$.
\end{corollary}
\begin{proof}
  Suppose $\dim(Y) = e$. Using Proposition \ref{prop:dense} and the same argument
  as in the proof of Proposition \ref{prop:mult} we may suppose that
  $Y = \bV(x_{e+1},\dots,x_n)$ since we are computing local
  multiplicities and they are preserved by the map $\phi_F$ constructed in
  Proposition \ref{prop:mult}. Now, the set of points in $Y_{\rm reg}$ at which
  the pair $(X_{\rm reg},Y_{\rm reg})$ fails to satisfy Whitney's condition (B) is
  contained in a proper Zariski-closed subset $W^*$ of $Y$, a suitable Zariski closed set $W^*$ is specified by Theorem \ref{thm:conorm} above, for example.
  Let $W$ be a Zariski closed set containing $Y_{\rm Sing}$ and $W^*$
  and set $U := Y- W$.  By \Cref{thm:multseq}, every point $y\in U$ has
  a euclidean neighborhood on which the sequence
  $m_{\bullet}(X, \bullet)$ is equal to $m_{\bullet}(X, y)$. As $Y$ is just an affine
  subspace od $ CC^n$, $U$ is connected, therefore
  $m_{\bullet}(X,\bullet)$ is constant on $U$, proving the theorem.
\end{proof}

Now we have the ability to (probabilistically) minimize a given
Whitney stratification. For a given variety $X\subset \CC^n$, the output of
$\mthcall{Whitney}(X)$ consists of a flag
\[W_0\subset W_1\subset \dots \subset W_d = X\] on $X$. Each $W_i$ is given by its
$\QQ$-irreducible components $W_{i1},\dots,W_{ir_i}$. Let us define
$W_{00}:=X$. These $W_{ij}$ may now be organized in a tree as follows:
Each node of this tree contains one of the $W_{ij}$. The children of a
node containing $W_{ij}$ are given by those $W_{i+1,k}$ with
$W_{i+1,k}\subset W_{ij}$. Let us call such a data structure a {\em Whitney
  tree} of $X$. Note that such a tree may be extracted from the output
of $\mthcall{Whitney}(X)$ by simple containment checks using Gröbner
bases. From such a tree, we may now minimize a given Whitney
stratification as follows:

\begin{center}
  \begin{tabular}{|r|l|}
    \hline
    ~ & {\bf WhitneyMinimize}$(\mathcal{T}_X)$ \\
    \hline
    ~&{\bf Input:} A Whitney tree $\mathcal{T}_X$ on $X$.\\
    ~&{\bf Output:} A minimal Whitney stratification of $X$.\\
    \hline
    1 & {\bf For} each node $W$ in $\mathcal{T}_X$\\
    2 & \spc $P\gets \setof{(W,Z,Y)}{Z, Y\text{ are nodes of }\mathcal{T}_X\text{ with }W\subset Z\subset Y}$\\
    3 & \spc {\bf If} $\mthcall{Mult}(Y,W) = \mthcall{Mult}(Y,Z)$ for all $(W,Z,Y)\in P$\\
    4 & \spc \spc Delete the node $W$ from $\mathcal{T}_X$\\
    5 & {\bf For} each level $i$ of $\mathcal{T}_X$\\
    6 & \spc $W_{d-i} \gets$ union of all components in level $i$\\
    7 & {\bf Return} $W_{\bullet}$\\
    \hline
  \end{tabular}
\end{center}

\begin{lemma}
  \mthcall{WhitneyMinimize} is correct and terminates. 
\end{lemma}
\begin{proof}
  The termination is clear. Suppose that we have a triple of varieties
  $W\subset Z\subset Y$ with $W$ and $Z$ irreducible and that for generic
  $w\in W$ and $z\in Z$ we have $m_{\bullet}(Y,w) = m_{\bullet}(Y,z)$. Then there
  exists a Zariski-dense subset $U$ of $Z$ such that
  $U\cap W\neq \emptyset$ and such that $m_{\bullet}(Y,\bullet)$ is constant on
  $U$. Hence the points in $W$ at which the pair $(Y,Z)$ does not
  satisfy Whitney's condition (B) have at least codimension one in $W$
  by \Cref{thm:multseq}. In the situation of
  \mthcall{WhitneyMinimize}, $W$ can thus be removed without
  destroying the property of the output flag $W_{\bullet}$ being a Whitney
  stratification. This proves the correctness.
\end{proof}

\begin{remark}
  \label{rem:opt}
  Note that we may also combine \mthcall{WhitneyMinimize} directly
  with \mthcall{WhitPoints}: For given $X$ and irreducible $Y$ let $h$
  be the output of \mthcall{WhitPoints}($X,Y$) and let
  $A:=Y\cap \bV(h)$. We know that the set of points in $Y$ where
  Whitney's condition (B) fails to hold with respect to $X$ and $Y$ is
  contained in $A$. We may then compute the irreducible components
  $A_1,\dots,A_r$ of $A$ and compare for each $A_i$ the output of
  \mthcall{Mult}$(X,Y)$ with \mthcall{Mult}$(X,A_i)$. If these
  sequences are equal we recursively go through the same procedure
  with $X$ and $A_i$. If they are not equal we append $A_i$ to a list
  of output components which we return in the end. This is potentially
  more optimal then minimizing after the stratification is finished
  because it prevents a build up of unneeded varieties.
\end{remark}

\section{Runtime Tests}\label{sec:runtimes}
In this section we collect some runtime comparisons of the new Whitney stratification algorithm described in this note with the algorithm of \cite{hnFOCM,hn2Real} for a variety of examples.  We also make note of if the resulting stratification is minimal and, if the stratification is not minimal, the time it takes to compute the coarsening to the minimal one. 

Note that in Table \ref{tab:example} we do not compare our implementation to the algorithms of \cite{dhinh2019thom} or of \cite{rannou1991complexity,rannou1998complexity} as all these algorithm are unable to successfully compute the Whitney stratification of the Whitney umbrella, the simplest non-trivial example, which is a surface in $\CC^3$ defined by the binomial equation $x^2-y^2z=0$. In both cases we have let these algorithms run for greater than 24 hours on the Whitney umbrella on our test machine and neither finished. Additionally, in the case of the quantifier elimination based method of  \cite{rannou1991complexity,rannou1998complexity}, the Maple implementation of quantifier elimination eventually gives a stack limit exceeded error even if we allow it to make the stack arbitrarily large on our test machine (which has an Intel Xeon W-3365 CPU with 1032 GB of RAM). For comparison the Whitney umbrella example takes 0.1s to run with the algorithm of \cite{hnFOCM,hn2Real} and 0.08s with the algorithm of  Whitney algorithm of \S \ref{sec:new}; both produce a minimal stratification. 

	\begin{table}[htb]
  \centering
  \resizebox{\linewidth}{!}{
  \begin{tabular}{lccccccc}
    \toprule
~&\multicolumn{2}{c}{\bf  Associated Primes Alg.~of \cite{hnFOCM,hn2Real}}  & \multicolumn{2}{c}{\bf Whitney Alg.~of \S \ref{sec:new}} & {\bf WhitneyMinimize} \\ \cmidrule(lr){2-3} \cmidrule(lr){4-5} \cmidrule(lr){6-6}\\
 ~ &Run time &Minimal & Run time &Minimal&  Run time &~\\
\cmidrule(r){1-1} 
  {\bf Input~~~~~}&\multicolumn{3}{c}{~}\\                                                 
           \midrule                                         
Whitney Cusp & 0.12s &Yes & 0.09s& Yes & 0.3s  \\
Example \ref{ex:IntroBiggerEx}& -- & -- & 7.3s& No & 140.5s  \\
$X_1$, see \eqref{eq:X1Ex}& 85.3s & Yes & 82.8s& Yes & 1.1s  \\
$X_2$, see \eqref{eq:X2Ex}& 4.8s & No & 4.6s& No & 68.1s  \\
$X_3$, see \eqref{eq:X3Ex}& 182.3s & -- & 192.7s& No & --  \\
$X_4$, see \eqref{eq:X4Ex}& 2.3s & Yes & 1.5s& Yes & 0.5s  \\
$X_5$, see \eqref{eq:X5Ex}& 1.3s & Yes & 0.9s& Yes & 2.4s  \\
$X_6$, see \eqref{eq:X6Ex}& -- & -- & 6080.8s& -- & --  \\
$X_7$, see \eqref{eq:X7Ex}& -- & -- & 15.4s& No & 20.2s  \\
$X_8$, see \eqref{eq:X8Ex}& -- & -- & 1177.6s& No & 1934.6s  \\
    \bottomrule
  \end{tabular}
  }
    \caption{Run times of several examples, -- denotes examples for which execution was stopped after 24 hours. These computations were run in version 1.24.5 of Macaulay2 on a workstation with an Intel Xeon W-3365 CPU and 1TB of RAM. }
  \label{tab:example}
\end{table}
The implementation of our Algorithm used to create these benchmarks is available in version $2.11$, and above, of the \texttt{WhitneyStratifications} Macaulay2  package, which is available on the first author's website at: \begin{center}
    \url{http://martin-helmer.com/Software/WhitStrat/WhitneyStratifications.m2}.
\end{center} With this Macaulay2  package loaded and given a polynomial ideal \texttt{I} our algorithm in Section 3 is called with \texttt{whitneyStratify(I, AssocPrimes=>false)}, while the algorithm of \cite{hnFOCM,hn2Real} is called with \texttt{whitneyStratify(I, AssocPrimes=>true)}. Given a stratification \texttt{W} output by one of these methods the {\bf WhitneyMinimize} procedure is run using the command \texttt{minCoarsenWS(W)}.

The Whitney cusp is defined by the equation
$ x_2^2+x_1^3-x_1^2x_3^2=0$. Below we list the defining equations of
the remaining examples presented in Table \ref{tab:example}. Examples
$X_3$ and $X_6$ are computations which arise when using map
stratification to identify the singularities of Feynman integrals in
quantum field theory in physics \cite{helmer2024landau}. While most
examples which did not finish were stopped after 8 hours, given the
longer run time of the example $X_6$ with the new algorithm we left
the associated primes algorithm running for over 7 days, without it
finishing. We note that the optimization described in Remark \ref{rem:opt}
is not part of our implementation. With this optimization we hope to
improve our runtimes further: On the example $X_3$ the stratification
produced by our algorithm includes a substantial amount of additional
strata compared to running the HN algorithm \cite{hnFOCM,hn2Real} on $X_3$, likely causing a
significant computational overhead due to a build up of redundant
varieties.
\begin{equation}
X_1=\bV\left(x_{1
     }^{2}x_{3}-x_{2}^{2},\,x_{2
     }^{4}-x_{1}x_{2}^{2}-x_{3}x
     _{4}^{2},\,x_{1}^{2}x_{2}^{
     2}-x_{1}^{3}-x_{4}^{2}\right)
\subset \CC^4
\label{eq:X1Ex}
\end{equation}
\begin{equation}
X_2=\bV\left(x_1^6+x_2^6+x_1^4x_3x_4+x_3^3\right)
\subset \CC^4
\label{eq:X2Ex}
\end{equation}

\begin{equation}
X_3=\bV\left(-p_{1}y_{1}y_{3}y_{4}-p_{1}y_{2}y_{3}y_{4}-s\,y_{1}y_{3}y_{5}-p_{2}y_{1}y_{4}y_{5}-t\,y_{2}y_{4}y_{5}\right)
\subset \CC^9
\label{eq:X3Ex}
\end{equation}

\begin{equation}
X_4=\bV(x_1^3-2x_1^2x_4+x_1^3+x_5^2x_4+x_2^2x_1-x_2x_1x_3+x_5^3)\subset \CC^5
\label{eq:X4Ex}
\end{equation}

\begin{equation}
X_5=\bV\left(x_{2}x_{5}^{2}-x_{1}^{2},\,x_{2}x_{4}x_{5}-x_{3}^{2}-x_{2},\,x_{1}^{2}x_{4}-x_{3}^{2}x_{5}-x_{2}x_{5}\right) \subset \CC^5
\label{eq:X5Ex}
\end{equation}

\begin{equation}
X_6=\bV(-p_{1}x_{1}x_{3}x_{4}-p_{1}x_{2}x_{3}x_{4}-s\,x_{1}x_{3}x_{5}-p_{2}x_{1}x_{4}x_{5}-t\,x_{2}x_{4}x_{5}-p_{1}x
     _{3}x_{4}x_{5})\subset \CC^9
\label{eq:X6Ex}
\end{equation}
\begin{equation}
  \label{eq:X7Ex}
  \begin{split}
    X_7=\bV(&x_1x_2^3 - 2x_1x_2^2x_4 + x_1x_2x_3x_5 + x_1x_2x_4^2 + x_1x_2x_5 - x_1x_3x_4x_5 - x_1x_4x_5 - x_2^2x_3x_5 -\\ & x_2^2x_4 - x_2^2x_5 + x_2x_3x_4x_5 - x_2x_3x_5 + 2x_2x_4^2 + x_2x_4x_5 - x_2x_5 - x_3^2x_5^2 + x_3x_4x_5 - \\ &2x_3x_5^2 - x_4^3 + x_4x_5 - x_5^2)
  \end{split}
\end{equation}
\begin{equation}
  \label{eq:X8Ex}
  \begin{split}
  X_8=\bV(&x_1^2x_2^2x_3 - 2x_1^2x_2x_3x_4 - 2x_1^2x_2x_3x_5 + x_1^2x_3x_4^2 + 2x_1^2x_3x_4x_5 + x_1^2x_3x_5^2- 2x_1x_2^3x_3 + \\ &4x_1x_2^2x_3x_4 + 4x_1x_2^2x_3x_5 - x_1x_2^2x_4 - 2x_1x_2x_3x_4^2 - 6x_1x_2x_3x_4x_5 - 2x_1x_2x_3x_5^2 + \\ &x_1x_2x_4^2 + x_1x_2x_4x_5 + 2x_1x_3x_4^2x_5 + 2x_1x_3x_4x_5^2 - x_1x_4^2x_5 + x_2^4x_3 - 2x_2^3x_3x_4 - \\ &2x_2^3x_3x_5 + x_2^3x_4 + x_2^2x_3x_4^2 + 4x_2^2x_3x_4x_5 + x_2^2x_3x_5^2 - 2x_2^2x_4^2 - x_2^2x_4x_5 - 2x_2x_3x_4^2x_5 - \\ & 2x_2x_3x_4x_5^2 + x_2x_4^3 + 2x_2x_4^2x_5 + x_3x_4^2x_5^2 - x_4^3x_5)
  \end{split}
\end{equation}

\bibliographystyle{abbrv}
	\bibliography{Whitney_Without_Primary_Decomp}

\begin{thebibliography}{10}

\bibitem{becker1993}
T.~Becker and V.~Weispfenning.
\newblock {\em Gr{\"o}bner Bases}, volume 141 of {\em Graduate Texts in
  Mathematics}.
\newblock {Springer-Verlag, New York}, 1993.

\bibitem{decker1999}
W.~Decker, G.-M. Greuel, and G.~Pfister.
\newblock Primary {{Decomposition}}: {{Algorithms}} and {{Comparisons}}.
\newblock In {\em Algorithmic {{Algebra}} and {{Number Theory}}}, pages
  187--220. {Springer}, 1999.

\bibitem{eisenbud1995}
D.~Eisenbud.
\newblock {\em Commutative Algebra: With a View toward Algebraic Geometry}.
\newblock {Springer New York}, {New York, NY}, 1995.

\bibitem{FTpolar}
A.~G. Flores and B.~Teissier.
\newblock Local polar varieties in the geometric study of singularities.
\newblock {\em Annales de la Facult\'e des sciences de Toulouse :
  Math\'ematiques}, Ser. 6, 27(4):679--775, 2018.

\bibitem{M2}
D.~R. Grayson and M.~E. Stillman.
\newblock {Macaulay2, a software system for research in algebraic geometry}.
\newblock Available at \url{http://www.math.uiuc.edu/Macaulay2}.

\bibitem{HH19}
C.~Harris and M.~Helmer.
\newblock Segre class computation and practical applications.
\newblock {\em Mathematics of Computation}, 89(321):465--491, 2020.

\bibitem{hnFOCM}
M.~Helmer and V.~Nanda.
\newblock Conormal spaces and whitney stratifications.
\newblock {\em Foundations of Computational Mathematics}, pages 1--36, 2022.

\bibitem{hnFOCMCorrection}
M.~Helmer and V.~Nanda.
\newblock Correction to: Conormal spaces and whitney stratifications.
\newblock {\em Foundations of Computational Mathematics}, pages 1--8, 2023.

\bibitem{hn2Real}
M.~Helmer and V.~Nanda.
\newblock Effective {W}hitney stratification of real algebraic varieties.
\newblock {\em arXiv:2307.05427v2}, 2023.

\bibitem{helmer2024landau}
M.~Helmer, G.~Papathanasiou, and F.~Tellander.
\newblock Landau singularities from whitney stratifications.
\newblock {\em arXiv preprint arXiv:2402.14787}, 2024.

\bibitem{dhinh2019thom}
S.~T. {\DH}inh and Z.~Jelonek.
\newblock Thom isotopy theorem for nonproper maps and computation of sets of
  stratified generalized critical values.
\newblock {\em Discrete \& Computational Geometry}, 65:279--304, 2021.

\bibitem{Mather2012}
J.~Mather.
\newblock {Notes on topological stability}.
\newblock {\em Bulletin of the American Mathematical Society}, 49(4):475--506,
  2012.

\bibitem{rannou1991complexity}
T.~Mostowski and E.~Rannou.
\newblock Complexity of the computation of the canonical whitney stratification
  of an algebraic set in $\mathbb{C}^n$.
\newblock In H.~F. Mattson, T.~Mora, and T.~R.~N. Rao, editors, {\em Applied
  Algebra, Algebraic Algorithms and Error-Correcting Codes}, pages 281--291,
  Berlin, Heidelberg, 1991. Springer Berlin Heidelberg.

\bibitem{parusinski2025}
A.~Parusiński and L.~Paunescu.
\newblock On teissier's example of an equisingularity class that cannot be
  defined over the rationals, 2025.

\bibitem{rannou1998complexity}
E.~Rannou.
\newblock The complexity of stratification computation.
\newblock {\em Discrete \& Computational Geometry}, 19(1):47--78, 1998.

\bibitem{sayrafi2017}
M.~{Sayrafi}.
\newblock {Computations over Local Rings in Macaulay2}.
\newblock {\em arXiv e-prints}, page arXiv:1710.09830, Oct. 2017.

\bibitem{sommeseNumericalSolutionSystems2005}
A.~J. Sommese and C.~W. Wampler~(II.).
\newblock {\em The {{Numerical Solution}} of {{Systems}} of {{Polynomials
  Arising}} in {{Engineering}} and {{Science}}}.
\newblock World Scientific, 2005.

\bibitem{tessier1981varieties}
B.~Tessier.
\newblock Polar vari{\'e}t{\'e}s 2: Polar multiplicity, plane sections, and
  whitney conditions.
\newblock In {\em Actes de la conference de g{\'e}ometrie alg{\'e}briquea la
  R{\'a}bida, Springer Lecture Notes}, volume 961, pages 314--491, 1981.

\bibitem{wallregular}
C.~Wall.
\newblock Regular stratifications.
\newblock In {\em Dynamical Systems—Warwick 1974: Proceedings of a Symposium
  Held at the University of Warwick 1973/74}, pages 332--344. Springer, 1974.

\bibitem{whitney1965tangents}
H.~Whitney.
\newblock Tangents to an analytic variety.
\newblock {\em Annals of mathematics}, pages 496--549, 1965.

\end{thebibliography}
\end{document}